\documentclass[11pt,leqno]{article}

\usepackage{amssymb,amsmath,amsthm}

\newcommand{\n}{\noindent}
\newcommand{\bb}[1]{\mathbb{#1}}
\newcommand{\cl}[1]{\mathcal{#1}}
\newcommand{\vp}{\varepsilon}

\theoremstyle{plain}
\newtheorem{thm}{Theorem}
\newtheorem*{cor}{Corollary}

\theoremstyle{definition}
\newtheorem{dfn}{Definition}
\newtheorem{prb}{Problem}
\newtheorem{exm}{Example}

\theoremstyle{remark}
\newtheorem*{rk}{Remark}

\begin{document}

\title{Random Series of Trace Class Operators\thanks{In \emph{Proceedings IV CLAPEM} (Mexico City, September 1990).}}

\author{by\\ Gilles Pisier\\
Texas A\&M University\\
College Station, TX 77843, U. S. A.\\
and\\
Universit\'e Paris VI\\
Equipe d'Analyse, Case 186, 75252\\
Paris Cedex 05, France}

\date{}
\maketitle

In this lecture, we would like to present some recent results on Gaussian (or Rademacher) random series of trace class operators, (\emph{cf.}\ mainly \cite{LPP}). We will emphasize the probabilistic reformulation of the results of \cite{LPP}, as well as the open problems suggested by these results. We will start by a survey of what is known about the following problems.

Let $B$ be a Banach space. Let $(x_n)$ be a sequence of elements of $B$. Let $(g_n)$ be a sequence of i.i.d.\ Gaussian random variables.

\begin{prb}\label{rans1prb}
Find a necessary and sufficient condition on the sequence $(x_n)$ for the a.s.\ convergence in norm of the series
\begin{equation}\label{rans1eq}
 \sum^\infty_{n=1} g_nx_n.
\end{equation}
\end{prb}

A similar question can be raised for series of the form $\sum \pm x_n$, where the signs are chosen at random. More precisely, let $(\vp_n)$ be a sequence of i.i.d.\ $\pm1$-value random variables such that $P(\vp_n=+1) = P(\vp_n=-1)=1/2$. We can ask

\begin{prb}\label{rans2prb}
Find a necessary and sufficient on condition $(x_n)$ for the a.s.\ convergence of the series
\begin{equation}\label{rans2eq}
 \sum^\infty_{n=1} \vp_nx_n.
\end{equation}
\end{prb}

Let $S_n = \sum\limits_{k\le n} g_kx_k$ and $R_n = \sum\limits_{k\le n} \vp_kx_k$. Let $0\le p\le \infty$. We will say that the sequence of partial sums $\{S_n\}$ converges in $L_p(B)$ if there is a Bochner-measurable $B$-valued random variable $S_\infty$ such that $\|S_n-S_\infty\|\to 0$ in $L_p$. For $p=0$, this corresponds to the convergence in probability of the series \eqref{rans1eq}. (Recall that the definition of a Bochner-measurable random variable is equivalent to requiring that the variable essentially takes its values in a separable subspace of $B$. Equivalently, the image probability measure on $B$ is a Radon measure.)

The following result is well known. In the Gaussian case, it goes back to \cite{F} and \cite{LS}, while the Rademacher case goes back to \cite{K} and \cite{Kw1} (\emph{cf.}\ also \cite{IN} for the case $p=0$).

\begin{thm}\label{rans1thm}
The following assertions are equivalent:
\begin{itemize}
\item[\rm (i)] The sequence of partial sums $\{S_n\}$ converges a.s.\ in the norm of $B$.
\item[\rm (ii)] The sequence $\{S_n\}$ converges in $L_p(B)$ for some $0\le p<\infty$.
\item[\rm (iii)] The sequence $\{S_n\}$ converges in $L_p(B)$ for all $0\le p<\infty$.
\end{itemize}

The same equivalence holds with the partial sums $\{R_n\}$ of the series \eqref{rans2eq} instead of the series \eqref{rans1eq}. Moreover, when the series \eqref{rans1eq} or \eqref{rans2eq} is convergent a.s.\ to a limit denoted by $S$, necessarily there is a $\delta>0$ such that
\[
 \int \exp(\delta\|S\|^2) < \infty.
\]
\end{thm}

\begin{cor}
For all $0<p<q<\infty$, the norms induced by $L_p(B)$ and $L_q(B)$ on the set of all a.s.\ convergent series of the form \eqref{rans1eq} $($resp.\ \eqref{rans2eq}$)$ are equivalent. $($In particular, they are all equivalent to the norm induced by $L_2(B)$.$)$
\end{cor}

\begin{rk}
Of course, the preceding result gives an answer to Problem~\ref{rans1prb}: \ the convergence in $L_2(B)$ for instance is necessary and sufficient for the a.s.\ convergence. However, this does not help us (except in the few simple cases below) since we usually cannot compute the norm in $L_2(B)$ for a general Banach space $B$. We are looking for a condition \emph{as simple as possible} characterizing the a.s.\ convergence of \eqref{rans1eq}.
\end{rk}

\begin{rk}
Recently, Talagrand \cite{T} gave a remarkable necessary and sufficient condition for the a.s.\ continuity of Gaussian processes. This can be rephrased as a solution of Problem~\ref{rans1prb} in full generality, as follows. Let $K$ be the closed unit ball of the dual $B^*$ equipped with the topology $\sigma(B^*,B)$ for which it is compact. The series \eqref{rans1eq} converges a.s.\ in $B$ iff the ``majorizing measure condition'' holds. The latter means that there is a positive finite measure $m$ on $K$ such that if we set $\forall t$, $s\in K$
\[
 d(t,s) = \left(\sum |\langle x_n,t-s\rangle|^2\right)^{1/2}
\]
and $\forall \vp>0$ $B(t,\vp)=\{s\in K\mid d(t,s)<\vp\}$ then the condition can be stated as
\[
 \lim_{\delta\to 0} \sup_{t\in K} \int^\delta_0 \left(\text{Log} \frac1{m(B(t,\vp))}\right)^{1/2} d\vp = 0.
\]
This completely solves the problem when we do not have any extra information on the Banach space $B$. However, as the discussion below will show, the spaces $B$ which arise in analysis are often given with an additional structure (for instance a function space, or an operator space $\ldots$) which allows to find a simple very explicit necessary and sufficient condition. For instance, it is not easy at all (although Talagrand did do it) to deduce the condition in Example~\ref{rans1exm} below (the Hilbert space case) from the majorizing measure condition. In this lecture we want to concentrate on Banach spaces $B$ which are as ``concrete'' as possible. This is somehow the other end of the spectrum from Talagrand's result which completely answers (and quite remarkably so) the ``abstract'' case.
\end{rk}

One interesting feature of Theorem~\ref{rans1thm} is that it allows us to define the Banach space of all a.s.\ convergent series of the form \eqref{rans1eq} or \eqref{rans2eq}. Indeed, given a fixed sequence $(g_n)$ of i.i.d.\ Gaussian normal variables on a probability space $(\Omega, {\cl A}, {\bb P})$, we define the space $G(B)$ as the subspace of $L_2(\Omega, {\cl A}, {\bb P}; B)$ formed by all the convergent series of the form \eqref{rans1eq}, and we equip it with the norm induced by $L_2(B)$. Similarly, we define the space $R(B)$ as the subspace of all convergent series of the form \eqref{rans2eq} in $L_2(B)$.

It is not hard to show that $G(B)$ (resp.\ $R(B)$) is a Banach space which coincides with the closure in $L_2(B)$ of all the \emph{finite} sums of the form \eqref{rans1eq} (resp.\ \eqref{rans2eq}).

By Theorem~\ref{rans1thm}, our Problems~\ref{rans1prb} and \ref{rans2prb} are the same as finding a ``simple'' description of the Banach spaces $G(B)$ and $R(B)$.

Let us first review several known cases to illustrate what we mean by a ``simple'' characterization.

\begin{exm}\label{rans1exm}
If $B$ is a Hilbert space, it is well known that \eqref{rans1eq} [resp.\ \eqref{rans2eq}] converges a.s.\ iff
\[
 \sum \|x_n\|^2 <\infty.
\]
Moreover, $G(B)$ (resp.\ $R(B)$) can be identified with $\ell_2(B)$ with equivalent norms.
\end{exm}

\begin{exm}\label{rans2exm}
If $B=L_p(S,\Sigma,\mu)$ for some measure space $(S,\Sigma,\mu)$ and if $1\le p<\infty$ (we insist that $p<\infty$) then the series \eqref{rans1eq} [resp.\ \eqref{rans2eq}] is a.s.\ convergent iff
\[
 \int \left(\sum_n |x_n(s)|^2\right)^{p/2} d\mu(s) < \infty.
\]
Moreover, there is a constant $C$ (depending only on $p$) such that
\[
 \frac1C \left(\int\left(\sum |x_n(s)|^2\right)^{p/2} d\mu(s)\right)^{1/p} \le  \left\|\sum g_nx_n\right\|_{L_2(B)} \le \left(\int\left(\sum |x_n(s)|^2\right)^{p/2} d\mu(s)\right)^{1/p},
\]
and similarly with the series \eqref{rans2eq} instead of \eqref{rans1eq}.
\end{exm}

In particular, if $B=L_p(\mu)$, the spaces $G(B)$ and $R(B)$ can be identified with the space $L_p(\mu;\ell_2)$, with equivalent norms. This has been known for a long time, as a consequence of classical inequalities. The ``modern'' way to prove this is to use Fubini's theorem and the Corollary of Theorem~\ref{rans1thm} and to work with the $L_p(B)$ norm on $G(B)$ or $R(B)$ instead of the $L_2(B)$-norm.

The basic examples above have been one of the motivation for the theory of type and cotype of Banach spaces, (\emph{cf.}\ \cite{MP}). A Banach space $B$ is called of type $p$ $(1\le p\le 2)$ if the condition $\sum\|x_n\|^p<\infty$ is sufficient for the a.s.\ convergence of the series \eqref{rans2eq}. Similarly, the space $B$ is called of cotype  $q$ $(2\le q<\infty)$ if the condition $\sum\|x_n\|^q<\infty$ is necessary for the a.s.\ convergence of the series \eqref{rans2eq}.

With this terminology, we find that (by Example~\ref{rans1exm}) if $B$ is isomorphic to a Hilbert space then $B$ is of type 2 and cotype 2. By a well known result of Kwapie\'n \cite{Kw2} the converse is also true, so that only in spaces isomorphic to Hilbert do we have as simple a characterization as in Example~\ref{rans1exm}. From Example~\ref{rans2exm}, it is rather easy to deduce (by H\"older--Minkowski) that $L_p$ is of cotype $\max(p,2)$ and, if $p\ne \infty$, of type $\min(p,2)$.

In our discussion, it is natural to ask for which spaces $B$ the series of the form \eqref{rans1eq} and \eqref{rans2eq} are equivalent in the sense that \eqref{rans1eq} converges a.s.\ iff \eqref{rans2eq} also does. This was answered in \cite{MP} as follows.

\begin{thm}\label{rans2thm}
The following properties of a Banach space $B$ are equivalent:
\begin{itemize}
\item[\rm (i)] The space $B$ is of cotype $q$ for some $q<\infty$.
\item[\rm (ii)] The a.s.\ convergence of a series of the form \eqref{rans1eq} is equivalent to the a.s.\ convergence of the corresponding series of the form \eqref{rans2eq}.
\item[\rm (iii)] The spaces $G(B)$ and $R(B)$ can be naturally identified.
\end{itemize}
\end{thm}

This is also equivalent to the nonexistence in $B$ of a constant $\lambda$ for which there exists a sequence of finite dimensional subspaces $B_n\subset B$ with $B_n$ $\lambda$-isomorphic to $\ell^n_\infty$. (Note that the natural inclusion $G(B)\subset R(B)$ holds for an arbitrary $B$, it is the converse inclusion which only holds if $B$ has a finite cotype.)

Since all the spaces that we will consider below fall into that category, we will sometimes only state our results for $R(B)$, which is more natural in view of the methods of proof, but the reader should recall that the same results holds for $G(B)$ as well.

\begin{exm}\label{rans3exm}
Let $B$ be a Banach lattice of cotype $q$ for some $q<\infty$. Without loss of generality, we can assume that $B$ is a Banach lattice of functions on some measure space $(S,\Sigma,m)$. Consider again $x_n\in B$. Then the series \eqref{rans2eq} converges a.s.\ if
\begin{equation}\label{rans3eq}
 \left(\sum^\infty_1 |x_n(\cdot)|^2\right)^{1/2} \in B.
\end{equation}
Moreover, there is a constant $C$ such that
\begin{equation}\label{rans4eq}
 \frac1C\left\|\left(\sum |x_n|^2\right)^{12/2}\right\|_B \le \left\|\sum \vp_nx_n\right\|_{L_2(B)} \le C\left\|\left(\sum|x_n|^2\right)^{1/2}\right\|_B.
\end{equation}
This useful result can be found (in a different but equivalent or more general formulation) in the work of Maurey \cite{M}. It gives us an identification of $R(B)$ and $G(B)$ with the space $B(\ell_2)$ of all sequences $(x_n)$ such that \eqref{rans3eq} holds equipped with the norm appearing on the left side of \eqref{rans4eq}. For a more thorough discussion of the many inequalities satisfied by a Banach lattice of finite cotype, (\emph{cf.}\ e.g.\ the first chapter in the book \cite{LT}). In another direction, we refer the reader to \cite{P6,PX} for a class of function spaces which, although they are not Banach lattices, satisfy an inequality analogous to \eqref{rans4eq}.
\end{exm}

Recently, F.~Lust-Piquard has obtained a striking \emph{noncommutative} version of Example~\ref{rans2exm} (and of Example~\ref{rans3exm} but we will not discuss this here, see \cite{LP2}). Although her results are valid in a more general framework, we will state them only for the simplest examples of noncommutative $L_p$-spaces, namely the Schatten $p$-classes $C_p$. For $1\le p<\infty$, we denote by $C_p$ the space of all compact operators $x\colon \ \ell_2\to \ell_2$ such that tr$|x|^p<\infty$, where $|x| = (x^*x)^{1/2}$. We equip this space with the norm
\[
 \|x\|_{C_p} = (\text{tr}|x|^p)^{1/p}.
\]
It is well known that tr$|x|^p = \text{tr}(x^*x)^{p/2} = \text{tr}(xx^*)^{p/2}$ for each $x$ in $C_p$.

We will now consider our main Problems~\ref{rans1prb} and \ref{rans2prb} with $B=C_p$.

The first part (Theorem~\ref{rans3thm}) seems perhaps less surprising at first glance than the second one (Theorem~\ref{rans4thm}), because it appears more similar to the commutative case, but see below for a clarification of the duality between the two cases.

\begin{thm}[\cite{LP1}]\label{rans3thm}
Assume $2\le q<\infty$. Let $B=C_q$ and let $x_n\in B$. Then the series \eqref{rans2eq} $($or \eqref{rans1eq}$)$ converges a.s.\ iff the series $\sum\limits^\infty_1 x^*_nx_n$ and $\sum\limits^\infty_1 x_nx^*_n$ are both convergent in the strong operator topology and satisfy
\begin{equation}\label{rans5eq}
 \text{\rm tr}\left(\sum x^*_nx_n\right)^{q/2} < \infty\quad \text{\rm and}\quad {\rm tr}\left(\sum x_nx^*_n\right)^{q/2} < \infty.
\end{equation}
Moreover, the norm in the space $R(B)$ $($or $G(B))$ is equivalent to the expression
\begin{equation}\label{rans6eq}
 \max\left\{\left\|\left(\sum^\infty_1 x^*_nx_n\right)^{1/2}\right\|_{C_q}, \left\|\left( \sum^\infty_1 x_nx^*_n\right)^{1/2}\right\|_{C_q}\right\}.
\end{equation}
\end{thm}

\begin{rk}
Here the two conditions appearing in \eqref{rans5eq} are \emph{not} equivalent and the theorem does \emph{not} hold if we drop one condition in \eqref{rans5eq}.
\end{rk}

The second part comes from \cite{LP1} for $1<p<2$ and from \cite{LPP} for $p=1$. Surprisingly, the condition is \emph{different} for the interval [1,2[ than for the interval $]2,\infty[$.

\begin{thm}[\cite{LP1,LPP}]\label{rans4thm}
Assume $1\le p\le 2$. Let $B=C_p$ and let $x_n\in B$. Then the series \eqref{rans2eq} $($or \eqref{rans1eq}$)$ converges a.s.\ iff there is a decomposition $x_n=y_n+z_n$ with $y_n,z_n\in C_p$ such that $\sum\limits^\infty_1 y^*_ny_n$ and $\sum\limits^\infty_1 z_nz^*_n$ converge in the strong operator topology and satisfy
\begin{equation}\label{rans7eq}
\text{\rm tr}\left(\sum y^*_ny_n\right)^{p/2} < \infty\quad \text{\rm and}\quad {\rm tr} \left(\sum z_nz^*_n\right)^{p/2} < \infty.
\end{equation}
Moreover, if we define
\begin{equation}\label{rans8eq}
|||(x_n)|||_p = \inf\left\{\left\|\left(\sum y^*_ny_n\right)^{1/2}\right\|_{C_p} + \left\|\left( \sum z_nz^*_n\right)^{1/2}\right\|_{C_p}\right\}.
\end{equation}
where the infimum runs over all possible decompositions $x_n=y_n+z_n$, then there is a constant $C$ $($depending only on $p)$ such that
\[
 \frac1C |||(x_n)|||_p \le \left\|\sum^\infty_1 \vp_nx_n\right\|_{R(C_p)} \le C|||(x_n)|||_p.
\]
\end{thm}

\begin{rk}
To illustrate the preceding result, let us consider the following illuminating special case given in \cite{LP1}. Let $(\vp_{ij})$ be a collection of i.i.d.\ symmetric $\pm1$ valued random variables as before, but this time indexed by ${\bb N}\times {\bb N}$ instead of ${\bb N}$.

Let $A = (a_{ij})$ be an infinite matrix with complex entries. Then, if $2\le q<\infty$, the random matrix $(\vp_{ij},a_{ij})$ is a.s.\ in $C_q$ iff we have both
\[
 \sum_i\left(\sum_j|a_{ij}|^2\right)^{q/2} < \infty \quad \text{and}\quad \sum_j \left(\sum_i |a_{ji}|^2\right)^{q/2} < \infty.
\]
 \end{rk}

Moreover, if $1\le p < 2$, then $(\vp_{ij}a_{ij})$ is a.s.\ in $C_p$ iff $a_{ij}$ can be decomposed as $a_{ij}=b_{ij}+c_{ij}$ with
\[
 \sum_i \left(\sum_j|b_{ij}|^2\right)^{p/2} < \infty\quad \text{and}\quad \sum_j \left(\sum_i |c_{ij}|^2\right)^{p/2} < \infty.
\]

(Note:\ the case $0<p<1$ does not seem to be known.)

\begin{rk}
The norm appearing in \eqref{rans8eq} is equivalent to the dual norm to the norm appearing in \eqref{rans6eq} when $\frac1p+\frac1q=1$. This is a special case of the general duality between the intersection of two spaces and the sum of their duals. However, the reader should be warned that \eqref{rans6eq} and \eqref{rans8eq} are \emph{not} equivalent norms when $p=q$ unless $p=q=2$!
\end{rk}

\begin{rk}
In the preceding examples, when $B=L_p$ or $C_p$, we can observe that if $1<p<\infty$ the dual of $G(B)$ (resp.\ $R(B)$) can be identified with $G(B^*)$ (resp.\ $R(B^*)$). This property has been extensively studied under the name of $K$-convexity. It means equivalently that there is a natural bounded linear projection from $L_2(B)$ onto $G(B)$ or $R(B)$. We refer the reader to \cite{P1,P5} for more details on that property. In particular, we proved (\emph{cf.}\ \cite{P5}) that $B$ is $K$-convex iff $B$ is of type $p$ for some $p>1$ or iff $B$ does not contain a sequence of finite dimensional subspaces uniformly isomorphic to $\ell^n_1$.
\end{rk}

In the case $p=1$, the space $C_1$ is the space of all trace class operators which can also be viewed as a tensor product. Let us recall the definition of the projective tensor product of two Banach spaces. Let $E,F$ be Banach spaces and let $u = \sum\limits^n_{i=1} x_i\otimes y_i\in E \otimes F$ with $x_i\in E$, $y_i\in F$. We define
\[
 \|u\|_\Lambda = \inf \left\{\sum^n_1 \|x_i\|\|y_i\|\right\}
\]
where the infimum runs over all possible representations of $u$ as a finite sum as above. The projective tensor product $E\widehat \otimes  F$ is defined as the completion of $E\otimes F$ for this norm. It is well known that its dual $(E\widehat\otimes F)^*$ can be identified with the space of all bounded bilinear forms on $E\times F$, or with the spaces of all bounded operators either $B(E,F^*)$ (from $E$ into $F^*$) or $B(F,E^*)$ (from $F$ into $E^*$). We will see below that for a certain class of Banach spaces $E$ and $F$, there is a rather simple characterization of the sequences $x_n\in E\widehat\otimes F$ such that the series \eqref{rans1eq} or \eqref{rans2eq} converges a.s.\ in $E\widehat\otimes F$. Quite surprisingly the only known proof (at the moment) of these results uses the factorization of operator valued analytic functions. In the case of $C_1$ this factorization goes back to Sarason \cite{S}, following classical work in the matrix case by Wiener--Masani and Helson--Lowdenslager. Recently, these classical results have been generalized in \cite{P3} to the setting of type 2 Banach spaces (\emph{cf.}\ also \cite{HP,P2,P4}). To explain more clearly the connection with series of independent random variables as in \eqref{rans1eq} and \eqref{rans2eq}, we present our results on the infinite dimensional  torus $\Delta = \pmb{T}^{\bb N}$, equipped with the probability measure $m$ which is the infinite product of copies of the normalized Haar measure on the one dimensional torus $\pmb{T}$. Let $B$ be a Banach space. Consider first a function $F\colon \ \pmb{T}\to B$ which is in $L^1(\pmb{T},dt;B)$. We will say that $F$ is \emph{analytic} if its Fourier transform vanishes on the negative integers.

Let $1\le p\le \infty$.
Now consider a function $f$ in $L_p(\Delta,m;B)$. Let us denote by $E_n$ the conditional expectation operator on $L_p(\Delta,m;B)$ with respect to the $\sigma$-algebra generated by the first $(n+1)$ coordinates $(t_0,t_1,\ldots, t_n)$ on the infinite dimensional torus $\Delta = \pmb{T}^{\bb N}$. We will say that $f$ is  a Hardy function if for each integer $n\ge 0$ the function $E_nf$ is analytic as a function of the last variable $t_n$, more precisely, the function $t\to E_0f(t)$ is analytic and for all $n>0$, for any fixed $(t_0,\ldots, t_{n-1})\in \pmb{T}^n$ the function $t\to E_nf(t_0,\ldots, t_{n-1},t)$ is analytic in the above sense.

We will denote by $H^p(\Delta,B)$ the closed subspace of $L_p(\Delta,m;B)$ formed by all the Hardy functions $f$. This is the same as the set of all Hardy martingales considered in \cite{E} and \cite{G}. These Hardy martingales are a convenient discretization of a certain kind of stochastic integrals which includes analytic functions of the complex Brownian motion, (\emph{cf.}\ e.g. \cite{C,V1,V2}).

We will need the following.

\begin{dfn}
Let $E$ be a Banach space. We will say that $E$ has the analytic UMD property if there is a constant $K$  such that for all $g$ in $H^2(\Delta,E)$ and for all choices of signs $\vp_n = \pm 1$ the series $\sum\limits^\infty_{n=1} \vp_n(E_ng-E_{n-1}g)$ converges in $H^2(\Delta,E)$ and satisfies
\begin{equation}\label{rans9eq}
 \left(\int\left\|\sum \vp_n(E_ng-E_{n-1}g)\right\|^2dm\right)^{1/2} \le K\left(\int \|g\|^2 dm\right)^{1/2}.
\end{equation}
We will denote by $K(E)$ the smallest constant $K$ such that this holds.
\end{dfn}

We observe for further use that if we denote by $\mu$ the normalized uniform probability measure on all the choices of signs $\vp = (\vp_n)_{n\ge 0} \in \{-1,1\}^{\bb N}$, then by averaging \eqref{rans9eq} over all choices of signs we find (denoting $d_ng=E_ng-E_{n-1}g$)
\[
 \iint\left\|\sum \vp_nd_ng(t)\right\|^2 dm(t) d\mu(\vp) \le K^2 \int \|g\|^2 dm,
\]
hence by Fubini and a simple invariance argument, we also have for all points $t\in \pmb{T}^{\bb N}$
\begin{equation}\label{rans10eq}
 \iint\left\|\sum \vp_n e^{-it_n}d_ng(t)\right\|^2 dm(t) d\mu(\vp)\le 4K^2 \int \|g\|^2 dm.
\end{equation}

This property is slightly weaker than the so-called UMD property for which we refer to Burkholder's paper \cite{B}. Here we only require that ``analytic'' martingales (meaning those associated to a Hardy function $g$) are unconditional in $L_2(\Delta,m;B)$. The UMD property corresponds to the same as \eqref{rans9eq} but for arbitrary $E$-valued martingales. We refer to \cite{G} for a discussion of the analytic UMD property. Note that the spaces $L_p$ or $C_p$ are UMD for $1<p<\infty$, and that $L_1$ has the analytic UMD property (see \cite{G}), while $C_1$ fails it (\emph{cf.}\ \cite{HP}).

\begin{rk}
To illustrate the possible applications of Theorems~\ref{rans3thm} and \ref{rans4thm}, let us consider their significance for martaingales with values in $C_q$. Let $(M_n)$ be a martingale with values in a UMD Banach space $B$. Let us denote as usual $dM_n=M_n-M_{n-1}$. We can define the vector valued version of the maximal function and of the square function, as follows
\[
 M^* = \sup_n \|M_n\|\quad \text{and}\quad S(M) = \left(\sup_n \int \left|\sum^n_1 \epsilon_k dM_k\right\|^2 d\mu(\epsilon)\right)^{1/2}.
\]
\end{rk}

\n Then, it is known (\emph{cf.}\ \cite{B}) that for any $1\le r\le \infty$, there is a constant $K_r$ such that for all $B$-valued martingales $(M_n)$ in $L_r(B)$, we have
\[
 (K_r)^{-1}\|S(M)\|_r \le \|M^*\|_r \le K_r\|S(M)\|_r.
\]
It was observed in \cite{Bo} (\emph{cf.}\ also \cite{RF} for more information) that the spaces $C_q$ are all UMD spaces when $1<q<\infty$. Thus, we can reformulate Theorems~\ref{rans3thm} and \ref{rans4thm} in a martingale setting as follows. Given a martingale $(M_n)$ with values in $C_q$, for any $\omega$ in our probability space, let $x_n = dM_n(\omega)$ and let us denote if $2\le q<\infty$ (resp.\ $1\le p\le 2$) by $S_q(M)(\omega)$ (resp.\ $S_p(M)(\omega)$) the expression \eqref{rans6eq} (resp.\ \eqref{rans8eq}). Then, we can state for all $1\le r,q<\infty$, there is a constant $K_{r,q}$ such that for all $C_q$-valued martingales $(M_n)$, we have
\[
 (K_{r,q})^{-1}\|S_q(M)\|_r \le \|M^*\|_r \le K_{r,q}\|S_q(M)\|_r.
\]

Since it is well known that $C_1$ can be identified with $\ell_2\widehat\otimes \ell_2$ it is natural to try to extend the case $p=1$ of Theorem~\ref{rans4thm} to more general projective tensor products. Indeed, we have

\begin{thm}\label{rans5thm}
Assume that $E,F$ are both Banach spaces of type 2 with the analytic UMD property. Let $B = E\widehat\otimes F$ and let $x_n \in E\widehat\otimes F$. Then the series \eqref{rans1eq} or \eqref{rans2eq} converges a.s.\ iff there are positive scalars $\lambda_m$ with $\sum \lambda_m < \infty$ and elements $y^m_n \in E$, $z^m\in F$, $y^m\in E$, $z^m_n\in F$ such that
\begin{equation}\label{rans11eq}
 x_n = \sum^\infty_{m=1} \lambda_m[y^m_n \otimes z^m + y^m \otimes z^m_n]
\end{equation}
such that
\[
 \sup_m \left\|\sum_n \epsilon_ny^m_n\right\|_{L_2(E)} \le 1, \quad \sup_n \left\|\sum_n \epsilon_nz^m_n\right\|_{L_2(F)} \le 1, \quad \sup_m\|y^m\|_E \le 1, \sup_m\|z^m\|_F\le 1.
\]
Equivalently, we have the isomorphic identification
\begin{equation}\label{rans12eq}
 R(E\widehat\otimes F) \approx R(E)\widehat\otimes F + E\widehat\otimes R(F),
\end{equation}
and the norm of $R(E\widehat\otimes F)$ is equivalent to the natural norm in the space $R(E) \widehat\otimes F + E\widehat\otimes R(F)$ which is defined as
\[
 |||(x_n)||| = \inf\left\{\|(y_n)\|_{R(E)\widehat\otimes F} + \|(z_n)\|_{E\widehat\otimes R(F)}\right\}
\]
where the infimum runs over all decompositions $x_n = y_n+z_n$. $($Another trivially equivalent norm can be defined as the infimum of $\sum \lambda_m$ over all possible representations as in \eqref{rans11eq} above.$)$
\end{thm}

\begin{rk}
The sum in \eqref{rans12eq} is \emph{not} a direct sum. The meaning of \eqref{rans12eq} is that the three spaces appearing in \eqref{rans9eq} are all naturally continuously injected into the set of all sequences of elements of $E\widehat\otimes F$ and a sequence $(x_n)$ comes from $R(E\widehat\otimes F)$ iff it can be decomposed as the sum of a sequence in $R(E)\widehat\otimes F$ and a sequence in $E\widehat\otimes R(F)$.
\end{rk}

For example, if we take $E=F=\ell_2$, then $R(E) \approx \ell_2(E)$ and $R(F)\approx \ell_2(F)$, and one can check rather easily that $\|(y_n)\|_{\ell_2(E)\widehat\otimes F}$ (resp.\ $\|(z_n)\|_{E\widehat\otimes \ell_2(F)}$) can be identified with
\[
 \left\|\left(\sum y^*_ny_n\right)^{1/2}\right\|_{E\widehat\otimes F} \left(\text{resp. } \left\|\left(\sum z_nz^*_n\right)^{1/2}\right\|_{E\widehat\otimes F}\right).
\]
Thus, Theorem~\ref{rans5thm} extends the case $p=1$ of Theorem~\ref{rans4thm}.

The proof of Theorem~\ref{rans5thm} is based on a factorization property of functions of $H^1(\Delta, E\widehat\otimes F)$ as a convex hull of tensor products of bounded sets of functions in $H^2(E)$ and $H^2(F)$. From the viewpoint of Harmonic Analysis the space $H^p(\Delta,B)$ is nothing but the $B$-valued case of the $H^p$-space associated to the compact group $\Delta$ with its dual group ${\bb Z}^{\bb N}$ ordered by the lexicographical order, as is explained for instance in the chapter devoted to compact groups with ordered duals in \cite{R}. It is easy to see that if $g,h$ are complex valued functions in $H^2(\Delta)$, then the pointwise product $gh$ is in $H^1(\Delta)$  with $\|gh\|_1\le \|g\|_2\|h\|_2$. Since the methods described in Rudin's book, (which are due to Helson--Lowdenslager) extend to the matrix valued case (\emph{cf.}\ \cite{HL}) they allow us to trivially modify the proofs of the Appendix~B in the paper \cite{P3} (\emph{cf.}\ also \cite{P4}) to obtain the following statement (which could alternatively be phrased as a factorization of $H^1$-functions with values in the space of all nuclear operators from $E^*$ into $F$).

\begin{thm}\label{rans6thm}
Let $E,F$ be Banach spaces for type 2. Then the natural product mapping
\[
 H^2(\Delta) \times H^2(\Delta)\to H^1(\Delta)
\]
defines canonically a $\underline{\text{surjective}}$ norm one mapping
\[
 Q_{E,F}\colon \ H^2(\Delta,E) \widehat\otimes H^2(\Delta,F)\to H^1(\Delta,E\widehat\otimes F).
\]
More explicitly, there is a constant $C$ such that for any $f$ in $H^1(\Delta, E\widehat\otimes F)$ there are functions $g_m\in H^2(\Delta,E)$ and $h_m\in H^2(\Delta,F)$ such that
\begin{equation}
 f(z) = \sum^\infty_1 g_m(z) \otimes h_m(z) \tag*{$\forall z\in \Delta$}
\end{equation}
and
\begin{equation}\label{rans13eq}
 \sum\|g_m\|_{H^2(\Delta,E)} \|h_m\|_{H^2(\Delta,F)} \le C\|f\|_{H^1(\Delta,E\widehat\otimes F)}.
\end{equation}
\end{thm}
 
Using this result, it is not hard to complete the
\begin{proof}[Proof of Theorem \ref{rans5thm}]
We first note that there is obviously a norm one inclusion
\[
 R(E)\widehat\otimes F + E\widehat\otimes R(F)\to R(E\widehat\otimes F).
\]
Equivalently, for every sequence $(x_n)$ of the form \eqref{rans11eq}, the series $\sum \vp_nx_n$ is in $R(E\widehat\otimes F)$ with norm $\|\sum \vp_nx_n\|_{R(E\widehat\otimes F)} \le 2 \sum|\lambda_n|$. This is immediate from the definitions. We will prove the converse assuming (as we clearly may) that $B$ is a complex Banach space. We will work with the complex version of the series \eqref{rans2eq}. (The corresponding series are sometimes called Steinhaus series.) Consider $x_n\in E\widehat\otimes F$ such that the series $S = \sum \vp_nx_n$ converges a.s.\ in $E\widehat\otimes F$ and consider the function $f\colon \ \Delta\to B$ defined by
\[
 f(t_0,t_1,\ldots) = \sum_{n\ge 0} e^{it_n}x_n.
\]
Then, it is easy to check (working with finite sums) that $f$ converges a.s.\ and that
\begin{equation}\label{rans14eq}
\int_\Delta \|f\| dm \le 2\|S\|_{L_1(E\widehat\otimes F)}.
\end{equation}
Moreover (since $f$ is a polynomial of degree one in each complex variable) $f$ is clearly in $H^1(\Delta,B)$. By Theorem~\ref{rans6thm}, we can find $g_m,h_m$ such that \eqref{rans13eq} holds.\\
Now let us denote
\[
 d_nf = E_nf-E_{n-1}f.
\]
A simple calculation shows that
\begin{align}
 x_n &= \int e^{-it_n} f(t_0,t_1,\ldots) dm(t)
= \sum_m \int e^{-it_n} g_m\otimes h_m dm(t)\\
&= \sum_m \int e^{-it_n} [d_ng_m\otimes h_m+g_m \otimes d_nh_m]dm(t)\label{rans15eq}. 
\end{align}

Let us set
\[
 y^m_n(t) = e^{-it_n}d_ng_m(t),\quad z^m(t)=h_m(t),\quad y^m(t) = g_m(t),\quad z^m_n(t) = e^{-it_n}d_nh_m(t).
\]

Then by \eqref{rans10eq} (the analytic UMD property) we have $\sum \vp_ny^m_n(t) \in L_2(\Delta, dm; R(E)$) with
\begin{equation}\label{rans16eq}
 \left(\int\left\|\sum \vp_ny^m_n(t)\right\|^2_{R(E)} dm(t)\right)^{1/2} \le 2K(E) \|g_m\|_{H^2(\Delta,E)},
\end{equation}
and $\sum\vp_nz^m_n(t) \in L_2(\Delta, dm; R(F))$ with
\begin{equation}\label{rans17eq}
 \left(\int\left\|\sum \vp_nz^m_n(t)\right\|^2_{R(F)} dm(t)\right)^{1/2} \le 2K(F) \|h_m\|_{H^2(\Delta,F)},
\end{equation}
therefore it follows from \eqref{rans15eq} that we have
\[
 \sum_n \vp_nx_n = \sum_n \int \left[\left(\sum_n \vp_ny^m_h(t)\right) \otimes z^m(t) + y^m(t) \otimes \left(\sum_n \vp_nz^m_n(t)\right)\right] dm(t)
\]
so that recalling \eqref{rans13eq}, \eqref{rans14eq}, \eqref{rans16eq} and \eqref{rans17eq} we find
\begin{align*}
 \left\|\sum \vp_nx_n\right\|_{R(E) \widehat\otimes F + E\widehat\otimes R(F)} &\le (2CK(E) + 2CK(F))\|f\|_{H^1(\Delta,E\widehat\otimes F)}\\
&\le 4(CK(E) + CK(F))\|S\|_{R(E\widehat\otimes F)}
\end{align*}
so that we obtain the announced converse inclusion
\[
 R(E\widehat\otimes F) \subset R(E)\widehat\otimes F + E\widehat\otimes R(F),
\]
and this concludes the proof.
\end{proof}

\n{\bf Final Remarks.}

\begin{itemize}
 \item[(i)] Note that Theorem \ref{rans5thm} applies in the case $E=L_p$, $F=L_q$ and $2\le p,q< \infty$. We do not know what happens for $1<p,q<2$ or even in the case $E=L_p$, $F=L_2$ for $1<p<2$.
\item[(ii)] Similarly it is easy to deduce from Theorem~\ref{rans5thm} that the space $L_p \widehat\otimes L_q$ is of cotype $\max(p,q)$ if $2\le p,q<\infty$. However it remains an open problem whether $L_p\widehat\otimes L_q$ is of finite cotype (cotype 2?) for $1<p<2$ and $1<q\le 2$. For the case $p=q=2$, the result goes back to \cite{TJ}.
\item[(iii)] The preceding Theorem~\ref{rans5thm} raises many other natural questions. One fascinating point is that \eqref{rans12eq} appears like a \emph{derivation} formula. It seems that $R$ operates on the tensor product exactly as a derivation does on a product of noncommuting objects.
\begin{quote}
 Therefore, it is natural to ask the following two questions:\\
{\bf Is there a similar result  for $R(E\widehat\otimes F\widehat\otimes G)$?}\\
The guess is that we should find the sum
\[
 R(E) \widehat\otimes F\widehat\otimes G + E\widehat\otimes R(F) \widehat\otimes G + E\widehat\otimes F\widehat\otimes R(G).
\]
The most interesting open case is the case $E=F=G=\ell_2$. It is a long standing open question whether
the space $E\widehat\otimes F\widehat\otimes G$ has finite cotype in this case. \\
We can also ask about ``second derivatives'':\\
{\bf Is there an analogous result for $R[R(E\widehat\otimes F)]$?}\\
Here the guess is that we should have a formula analogous to the second derivative. 
\end{quote}
We must distinguish the first $R$ (associated to a first sequence $(\epsilon^1_n)$) and the second one (associated to a second sequence $(\epsilon^2_n)$, independent of the first one) so let us denote them by $R_1$ and $R_2$ respectively. Then the guess is that we should have
\[
 R_1(R_2(E\widehat\otimes F)) \approx R_1(R_2(E)) \widehat\otimes F + R_2(E) \widehat\otimes R_1(F) + R_1(E)\widehat\otimes R_2(F) + E\widehat\otimes R_1(R_2(F)).
\]
Even in the case $E=F=\ell_2$, we could not check this.
\end{itemize}


\begin{thebibliography}{10}
\bibitem{Bo}
J. Bourgain, Vector valued singular integrals and the $H^1$---BMO duality, \emph{Probability Theory and Harmonic Analysis}, Chao and Woyczynski (eds.), Marcel Dekker, (1986), pp.~1--19.
\bibitem{B}
D. Burkholder, A geometric characterization of Banach spaces in which martingale difference sequences are unconditional, \emph{Ann. Probab.} {\bf 9} (1981), 997--1011.
\bibitem{C}
K. Carne, The algebra of bounded holomorphic margingales, \emph{J. Funct. Anal.} {\bf 45} (1982), 95--108.
\bibitem{E}
G. Edgar, Analytic martingale convergence, \emph{J. Funct. Anal.} {\bf 69} (1986), 268--280.
\bibitem{F}
X. Fernique, Int\'egrabilit\'e des vecteurs gaussiens, \emph{C.R. Acad. Sci. Paris A} {\bf 270} (1970), 1698--1699.
\bibitem{G}
D.J.H. Garling, On martingales with values in a complex Banach space, \emph{Proc. Cambridge Phil. Soc.} {\bf 104} (1988), 399--406.
\bibitem{HL}
H. Helson and D. Lowdenslager, Prediction theory and Fourier series in several variables, \emph{Acta Math.} {\bf 99} (1958), 165--202.
\bibitem{HP}
U. Haagerup and G. Pisier, Factorization of analytic functions with values in noncommutative $L_1$-spaces, \emph{Canadian J. Math.} {\bf 41} (1989), 882--906.
\bibitem{IN}
K. Ito and M. Nisio, On the convergence of sums of independent Banach space valued random variables, \emph{Osaka J. Math.} {\bf 5} (1968), 35--48.
\bibitem{K}
J.P. Kahane, Some random series of functions, (1968), \emph{Heath Mathematical Monographs}, Second Edition, Cambridge University Press, (1985).
\bibitem{Kw1}
S. Kwapie\'n, A theorem on Rademacher series with vector valued coefficients, in \emph{Probability in Banach Spaces}, Springer Lecture Notes, no.~526, (1976), 157--158.
\bibitem{Kw2}
S. Kwapie\'n, Isomorphic characterizations of inner product spaces by orthogonal series with vector coefficients, \emph{Studia Math.} {\bf 44} (1972), 583--595.
\bibitem{LS}
H. Landau and L. Shepp, On the supremum of a Gaussian process, \emph{Sankhya A} {\bf 32} (1970), 369--378.
\bibitem{LT}
J. Lindenstrauss and L. Tzafriri, \emph{Classical Banach Spaces II}, Springer Verlag, 1979.
\bibitem{LP1}
F. Lust-Piquard, In\'egalit\'es de Khintchine dans $C_p$, $(1<p<\infty)$, \emph{C.R. Acad. Sci. Paris} {\bf 303} (1986), 289--292.
\bibitem{LP2}
F. Lust-Piquard, A Grothendieck factorization theorem on $2$2-convex Schatten spaces.  
Israel J. Math. 79 (1992),   331Ð365. 
\bibitem{LPP}
F, Lust-Piquard and G. Pisier, Noncommutative Khintchine and Paley inequalities, \emph{Ark. Mat.} {\bf 29} (1991), no. 2, 241--260.
\bibitem{M}
B. Maurey, Type et cotype dans les espaces munis de structure locale inconditionnelle, \emph{S\'eminaire Maurey--Schwartz} 73-75, Expos\'e no.~24--25, Ecole Polytechnique, Paris.
\bibitem{MP} 
B. Maurey and G. Pisier, S\'eries de variables al\'eatoires vectorielles ind\'ependantes et propri\'et\'es g\'eom\'etriques des espaces de Banach, \emph{Studia Math.} {\bf 58} (1976), 45--90.
\bibitem{P1}
G. Pisier, Probabilistic methods in the geometry of Banach spaces, CIME Summer School, June 1985, \emph{Springer Lecture Notes}, no. 1206, (1986), 167--241.
\bibitem{P2}
G. Pisier, Factorization de fonctions analytiques \`a valeurs op\'erateurs, \emph{C.R. Acad. Sci. Paris} {\bf 307} (1988), 955--960.
\bibitem{P3}
G. Pisier, Factorization of operator valued analytic functions, \emph{Adv. Math.} {\bf 93} (1992),  61--125.
\bibitem{P4}
G. Pisier, Factorization of operator valued analytic functions and complex interpolation, \emph{Festschrift in honor of I. Piatetski-Shapiro, Part II} S. Gilbert, R. Howe and P. Sarnak (eds.), Weizmann Science Press of Israel (1990), pp.~197--220.

\bibitem{P5}
G. Pisier, Holomorphic semi-groups and the geometry of Banach spaces,
\emph{Ann. Math.} {\bf 115} (1982), 375--392
\bibitem{P6}
G. Pisier, The dual $J^*$ of the James space has cotype 2 and the
Gordon--Lewis property, \emph{Proc. Cambridge Phil. Soc.} {\bf 103}
(1988), 323--331.
\bibitem{PX}
G. Pisier and Q. Xu, Random series in the real interpolation spaces
between the spaces $V_p$, \emph{GAFA, Springer Lecture Notes}, no. 1267,
(1987), 185--209.
\bibitem{RF}
J.L. Rubio de Francia, Martingale and integral  transforms of Banach
space valued functions, \emph{Springer Lecture Notes}, no. 1221, (1985),
195--222.
\bibitem{R}
W. Rudin, \emph{Fourier Analysis on Group}, Interscience, New York, 1962.

\bibitem{S}
D. Sarason, Generalized interpolation in $H^\infty$, \emph{Trans. Amer.
Math. Soc.} {\bf 127} (1967), 179--203.
\bibitem{T}
M. Talagrand, Regularity of Gaussian processes, \emph{Acta Math.} {\bf
159} (1987), 99--149.
\bibitem{TJ}
N. Tomczak-Jaegermann, On the moduli of convexity and smoothness and the
Rademacher averages of the trace classes $S_p$ $(1\le p < \infty)$,
\emph{Studia Math.} {\bf 50} (1974), 163--182.
\bibitem{V1}
N. Varopoulos, The Helson-Szeg\"o theorem and $A_p$-functions for
Brownian motion and several variables, \emph{J. Funct. Anal.} {\bf 39}
(1980), 85--121.
\bibitem{V2}
N. Varopoulos, Probabilistic approach to some problems in complex
analysis, \emph{Bull. Sci. Math.} {\bf 105} (1981), 181--224.



 

\end{thebibliography}
\end{document}